\documentclass[11pt]{amsart}
\usepackage[latin1]{inputenc}
\usepackage{amssymb}
\usepackage{enumerate}
\usepackage[english]{babel}

\usepackage[T1]{fontenc}
\usepackage{mathptmx}
\usepackage{amsthm}
\usepackage[colorlinks, bookmarks=true]{hyperref}
\usepackage{cleveref}

\newtheorem{theorem}{Theorem}[section]
\newtheorem{proposition}[theorem]{Proposition}

\theoremstyle{definition}
\newtheorem{remark}[theorem]{Remark}
\newtheorem{definition}[theorem]{Definition}

\numberwithin{equation}{section}

\crefname{theorem}{theorem}{theorems}
\Crefname{theorem}{Theorem}{Theorems}

\crefname{lemma}{lemma}{lemmas}
\Crefname{lemma}{Lemma}{Lemmas}

\crefname{proposition}{proposition}{propositions}
\Crefname{proposition}{Proposition}{Propositions}

\crefname{corollary}{corollary}{corollaries}
\Crefname{corollary}{Corollary}{Corollaries}

\newcommand{\unit}{\mathbf{1}}

\begin{document}
	
\title[The Quasi-linearity problem for Jordan Banach algebras]{The Quasi-linearity problem for Jordan Banach algebras: a topological characterization}

\author[G.M. Escolano]{Gerardo M. Escolano}

\address{Departamento de An{\'a}lisis Matem{\'a}tico, Facultad de Ciencias, Universidad de Granada \hyphenation{Gra-nada}, 18071 Granada, Spain.}
\email{gemares@ugr.es}



\begin{abstract} 
	Let $\mathfrak{J}$ be a JB$^*$-algebra with no quotients isomorphic to any spin factor. Let $\mu$ be a local quasi-linear Jordan functional on $\mathfrak{J}_{sa}$. We show that $\mu$ is a linear functional on $\mathfrak{J}_{sa}$ if and only if the restriction of $\mu$ to the closed unit ball of $\mathfrak{J}_{sa}$ is uniformly weakly continuous.
\end{abstract}

\maketitle

\section{Introduction}

The problem of quasi-linear mappings on C$^*$-algebras and von Neumann algebras is a longstanding question in functional analysis. J. F. Aarnes first considered this problem in \cite{Aarnes1969, Aarnes1970}, acknowledging R. V. Kadison for bringing the question to his attention, as a starting point to study whether a \emph{physical state} on a unital C$^*$-algebra $A$, i.e., a complex functional that acts as a positive linear functional of norm one on each singly generated C$^*$-subalgebra, is necessarily linear on $A$. In its proper physical interpretation, this corresponds to the problem of the expectation functional on the algebra of observables in quantum mechanics, c.f. Mackey \cite{Mackey1957}. In \cite{Aarnes1969, Aarnes1970} two facts are noted: the first one that the answer is no, even for abelian C$^*$-algebras ; the second one is that the problem was a special case of the famous Mackey-Gleason problem for measures on the lattice of projections of a von Neumann algebra \cite{Gleason1957}. \smallskip

He then reformulated the question so that it remained meaningful in the physical context of observables, represented by self-adjoint elements of $A$, by defining a \emph{quasi-state} $\rho$ to be a physical state with the further property that $\rho(a + ib) = \rho(a) + i\rho(b)$ for every $a, b \in A_{sa}$, the set of self-adjoints elements of $A$. He then proved that any quasi-state is linear when $A$ is abelian and that, under the extra assumption that $\rho$ is continuous on $A$, $\rho$ is linear for a quite large classes of C$^*$-algebras.  

This problem was closely related to the famous Mackey-Gleason problem for von Neumann algebras, since it is a special case of this linearity problem for physical states. 
An outstanding result by L. J. Bunce and J. D. M. Wright in \cite{BunceWright1992} and \cite{BunceWright1994} shows a positive solution to the Mackey-Gleason problem of signed measure on the lattice of a von Neumann algebra without central summands of type $I_2$. Moreover, in \cite{BunceWright1996} the authors try to apply the new tools developed for this particular case in order to obtain a positive solution to the problem of quasi-linear functionals on C$^*$-algebras.

The foundational works \cite{Jordan1933} and \cite{JorvNeuWign34} of Jordan, von Neumann, and Wigner suggest that the most general mathematical framework for quantum mechanics is not based on associative algebras, but rather on \emph{Jordan algebras}. These structures arise naturally from associative algebras through the symmetrization of the underlying associative product.  

The aim of this paper is to study quasi-linear functionals in the broader setting of JB$^*$-algebras, which can be regarded as the Jordan analogues of C$^*$-algebras. Motivated by recent results on the Mackey-Gleason-Bunce-Wright problem obtained in \cite{EscolanoPeraltaVillena2025}, it is natural to address the problem of quasi-linearity within the Jordan framework. More precisely, our goal is to provide a topological characterization of quasi-linear functionals in this setting, in the spirit of \cite{BunceWright1996}.\smallskip 

The first section of this paper is devoted to introducing the basic notions and properties of Jordan algebras that will be needed in the sequel, as well as reviewing key structural aspects of type $I_2$ summands in JBW$^*$-algebras, which play a central role throughout the paper. In the second section, we present the main result of the present work, Theorem \ref{Theo: 2.3}.

\subsection{Background }\label{subsec:background} \ \\

This part is entirely devoted to recall some structure theory of JB$^*$-algebras, JBW$^*$-algebras and JB$^*$-triples. We begin by recalling that a (real or complex) \emph{Jordan algebra} is a (real or complex) linear space $\mathfrak{J}$ equipped with a bilinear mapping $(a,b) \rightarrow a\circ b$ (called the Jordan product) satisfying that 
\begin{itemize}
    \item [(J1)] $a\circ b = b\circ a$, for all $a,b \in \mathfrak{J}$;
    \item [(J2)] $a^2\circ (b \circ a) = (a^2\circ b)\circ a$ for all $a,b \in \mathfrak{J}$ (Jordan identity).
\end{itemize}
Moreover, if $\mathfrak{J}$ is a (real or complex) Banach space satisfying 
\begin{itemize}
    \item [(J3)] $\|a\circ b\| \leq \|b\| \|a\|$, for all $a,b \in \mathfrak{J}$,
\end{itemize}
$\mathfrak{J}$ is called a (real or complex) \emph{Jordan Banach algebra}. A Jordan algebra $\mathfrak{J}$ is called unital if there exists an element $\unit \in \mathfrak{J}$ (called the
\emph{unit} of $\mathfrak{J}$) such that $\unit \circ a = a$ for all $a \in \mathfrak{J}$. Several properties of Jordan algebras are characterized in term of $U$-mappings. Let us define these special operators in the Jordan setting. Given elements $a,c \in \mathfrak{J}$, the symbol $U_{a,c}$ will stand for the linear mapping on $\mathfrak{J}$ defined by $$U_{a,c}(b):= (a\circ b)\circ c + (b\circ c)\circ a - (a\circ c)\circ b, \quad (b \in \mathfrak{J}.$$
We simply write $U_a$ for $U_{a,a}.$ For each $a \in \mathfrak{J}$ we denote by $M_a$ the Jordan multiplication operator by the element $a$, that is, $M_a: \mathfrak{J} \rightarrow \mathfrak{J}, M_a(b) = a\circ b$.\smallskip 

There are two closely related types of Jordan-Banach algebras, which are defined by algebraic-geometric axioms, and are known as JB-algebras and JB$^*$-algebras. A JB-algebra is a real Jordan-Banach algebra $\mathfrak{J}$ in which the norm satisfies the following two additional conditions:

\begin{itemize}
    \item[(JB1)] $\|a^2\| = \|a\|^2$ for all $a \in \mathfrak{J}$;
    \item[(JB2)] $\|a^2\| \leq \|a^2 + b^2\|$ for all $a,b \in \mathfrak{J}$.
\end{itemize}

The Jordan analogue to C$^*$-algebras is constituted by JB$^*$-algebras, a model introduced by I. Kaplansky in 1976. A complex Jordan-Banach algebra $\mathfrak{J}$ equipped
with an involution $*$ is said to be a JB$^*$-algebra if the following axiom is satisfied:

\begin{itemize}
    \item[(JB*1)] $\|U_a(a^*)\| = \|a\|^3$ for all $a \in \mathfrak{J}$.
\end{itemize}

JB- and JB$^*$-algebras are mutually linked in the following way: the set $\mathfrak{J}_{\mathrm{sa}}$ of all self-adjoint elements in a JB$^*$-algebra $\mathfrak{J}$, i.e. $\mathfrak{J}_{sa} := \{ a \in \mathfrak{J} : a^* = a \}$,
is a JB-algebra (see \cite[Proposition 3.8.2]{HOS}). Conversely, by a deep result due to J.\,D.\,M. Wright \cite{Wright1977}, each JB-algebra corresponds to the self-adjoint part of a (unique) JB$^*$-algebra.\smallskip

Every C$^*$-algebra is a JB$^*$-algebra if we consider the Jordan product given by $a\circ b = \frac12 (ab + ba)$. However, the class of JB$^*$-algebras is strictly larger than the class of C$^*$-algebras. A JBW$^*$-algebra (resp., a JBW-algebra) is a JB$^*$-algebra (resp., a JB-algebra) which is also a dual Banach space. Each JBW$^*$-algebra admits a unique (isometric) predual (cf. \cite[Theorem 4.4.16]{HOS} or \cite[Theorem 2.55]{AlfsenShultz2003}). Thus, JBW$^*$-algebras can be considered as the Jordan analogue of von Neumann algebras. It is known that a JB$^*$-algebra $\mathfrak{J}$ is a JBW$^*$-algebra if and only if its self-adjoint part $\mathfrak{J}_{\mathrm{sa}}$ is a JBW-algebra \cite{Edwards1980}. In \cite[\S 2]{AlfsenShultz2003} and \cite[Lemma 2.3]{Edwards1980}, it is shown that the second dual of any JB$^*$-algebra $\mathfrak{J}$, denoted by $\mathfrak{J}^{**}$, is a JBW$^*$-algebra under the Arens product (see \cite{Edwards1980} for the definition), which contains the original structure as a JB$^*$-subalgebra \cite[Theorem 4.4.3]{HOS}. \smallskip

The product of a Jordan algebra $\mathfrak{J}$ is, by definition, commutative, so all pairs of elements in $\mathfrak{J}$ commute if we only employ the ``usual'' sense. However, if an associative algebra $A$ is equipped with its natural Jordan product $a\circ b = \frac12 (ab + ba)$, it is straightforward to check that if $a,b\in A$ commute with respect to the associative product (i.e., $ab = ba$), then the corresponding Jordan multiplication operators $M_a$ and $M_b$ commute. According to standard sources, elements $a$ and $b$ in a Jordan algebra $\mathfrak{J}$ are said to \emph{operator commute} if the operators $M_a$ and $M_b$ commute (i.e., $(a\circ c)\circ b = a \circ (c\circ b)$ for every $c \in \mathfrak{J}$). For example, the Jordan identity is equivalent to saying that every element $a$ in a JB$^*$-algebra operator commutes with its square $a^2$. The general lack of associativity motivates the use of the \emph{associator} of three elements $a,c,b \in \mathfrak{J}$, defined by $[a,c,b] := (a\circ c) \circ b - a \circ (c\circ b) = [M_b, M_a](c)$. Clearly, $a$ and $b$ operator commute in $\mathfrak{J}$ if and only if $[a,\mathfrak{J},b]=0$. \smallskip

The \emph{centre} of a $JB^*$-algebra $\mathfrak{J}$ is formed by the elements $z\in \mathfrak{J}$ that operator commute with every element in $\mathfrak{J}$. The symbol $Z(\mathfrak{J})$ will stand for the centre of $\mathfrak{J}$, and its elements are called central. The centre of a $JB^*$-algebra $\mathfrak{J}$ is a commutative $C^*$-algebra, and contains the identity of $\mathfrak{J}$ in case the latter exists (see \cite[Proposition 1.52]{AlfsenShultz2003}). The centre of a $JBW^*$-algebra is a commutative von Neumann algebra (see \cite[Proposition 2.36]{AlfsenShultz2003}, \cite{Edwards1980}). If $Z(\mathfrak{J}) = \mathbb{C} 1$, we say that $\mathfrak{J}$ is a $JBW^*$-factor. \smallskip

An element $p$ in a JBW$^*$-algebra $\mathfrak{J}$ is called a \emph{projection} if $p^* = p = p\circ p$, i.e. $p$ is a projection in $\mathfrak{J}$ if and only if $p$ is a projection in $\mathfrak{J}_{sa}$. We denote by $\mathcal{P}(\mathfrak{J})$ the orthocomplemented lattice of all projections in $\mathfrak{J}$, equipped with the natural partial order $p \le q$ if and only if $p \circ q = p$. Two projections $p, q \in \mathcal{P}(\mathfrak{J})$ are said to be \emph{orthogonal}, written $p \perp q$, if $p \circ q = 0$. For each projection $p \in \mathcal{P}(\mathfrak{J})$, the smallest central projection $c \in \mathcal{P}(Z(\mathfrak{J}))$ satisfying $p \le c$ is called the \emph{central cover} (or \emph{central support}) of $p$, and it is denoted by $c(p)$ \cite[Definition~4.3.3]{HOS}. A projection $p \in \mathcal{P}(\mathfrak{J})$ is called \emph{abelian} if the JBW$^*$-subalgebra 
$$\mathfrak{J}_p := U_p(\mathfrak{J}) = \{U_p(x) : x \in \mathfrak{J}\}$$
is associative, that is, if $\mathfrak{J}_p = Z(\mathfrak{J}_p)$ \cite[Section~5.1.4]{HOS}. \smallskip

By endowing any JB$^*$-algebra $\mathfrak{J}$ with the \emph{canonical Jordan triple product} given by $\{a,b,c\}:=(a\circ b^*)\circ c + (c\circ b^*)\circ a - (a\circ c)\circ b^*$  for any $a,b,c \in \mathfrak{J}$, it has the structure of a \emph{JB$^*$-triple}. We refer to the class of JB$^*$-triples introduced by W. Kaup in \cite{Kaup1983}, i.e. A JB$^*$-triple is a complex Banach space $E$ together with a continuous triple product $\{\cdot, \cdot, \cdot\}: E\times E \times E \rightarrow E$, which is conjugate linear in the middle variable and symmetric bilinear in the outer variables satisfying that  

\begin{itemize}
    \item[(JB*T1)] $L(a,b)L(x,y) =L(x,y)L(a,b) + L(L(a,b)x,y)-L(x,L(a,b)y)$ for all $a,b,x,y \in E$, where $L(a,b)$ is the operator on $E$ given by $L(a,b)x= \{a,b,x\}$ (\emph{Jordan identity});
    \item [(JB*T2)] $L(a,a)$ is a Hermitian operator with non-negative spectrum for every $a \in E$;
    \item [(JB*T3)]$\|L(a,a)\| = \|a\|^2$ for every $a \in E$. 
\end{itemize}

An element $u$ in a JB$^*$-triple $E$ is called a \emph{tripotent} if $\{u,u,u\} = u$. Every tripotent $u \in E$ induces a decomposition (called the \emph{Pierce decomposition})
$$ E = E_2(u)\oplus E_1(u)\oplus E_0(u),$$
where for $i = 0,1,2$, $E_i(u)$ is the $i/2$ eigenspace of $L(u,u)$. The \emph{Pierce rules} are that $\{E_i(u), E_0(u), E\} \subset E_{i-j+k}(u)$ if $i-j+k \in \{0,1,2\}$ and is zero otherwise. The corresponding \emph{Pierce projections} $P_i(u): E \rightarrow E_i(u)$ $(i = 0,1,2)$ are contractive and satisfy 
$$P_2(u) = L(2L-Id), \quad P_1(u) = 4L(Id-L), \quad \text{and} \quad  P_0(u) = (I_d -L)(Id -2L), $$
where $L$ is the operator $L(u,u)$ and $Id$ is the identity map on $E$. It is also known that the Pierce-2 subspace $E(u)$ is a unital JB$^*$-algebra with unit $u$. A non-zero tripotent $u\in E$ is called \emph{minimal} (respectively, \emph{complete} or \emph{maximal}) if and only if $E_2(u) = \mathbb{C}u$ (respectively, $E_0(u)=0$).  \smallskip 

A JBW$^*$-triple is a JB$^*$-triple which is also a dual Banach space (with a unique predual \cite{BartonTimoney1986}). The second dual of a JB$^*$-triple is a JBW$^*$-triple \cite[Theorem 3.7]{Dineen1986}. Given any norm-one element $x$ in a JB$^*$-triple, the element $r(x)$, called the \emph{range tripotent} of $x$, is the smallest tripotent $u$ in $E^{**}$ satisfying that $x$ is positive  in the JBW$^*$-algebra $E_2^{**}(u)$.\smallskip

Two elements $a,b$ in a JB$^*$-triple $E$ are said to be \emph{orthogonal} (written $a \perp b$) if $L(a,b) = 0$. It is known (see \cite[Lemma 1]{BurgosPeralta2008}) that for $a,b \in E$, the following are equivalent 
$$
\begin{matrix}
    a \perp b; & \quad \{a,a,b\} = 0; \quad & a \perp r(b) \\  
    r(a) \perp r(b); & \quad E_2^{**}(r(a)) \perp E_2^{**}(r(b)); \quad & r(a) \in E_0^{**}(r(b)); \\
    a \in E_0^{**}(r(b)); & \quad b \in E_0^{**}(r(a)). 
\end{matrix}
$$

We shall say that a tripotent $e$ in a JB$^*$-triple $E$ has \emph{rank-$n$} if it can be written as the sum of $n$ orthogonal minimal tripotents in $E$. A subset $S$ of $E$ is called orthogonal if $0 \notin S$ and $x \perp y$ for every $x \neq y$ in $S$. The minimal cardinal number $r$ satisfying $\text{card}(S) \leq r$ for every orthogonal subset $ S \subseteq E$ is called the rank of $E$ (cf.\cite{BunceChu1992,Kaup1997}). \smallskip 

It follows by the Pierce rules that for each tripotent $u$ in a JB$^*$-triple $E$, the sets $E_0(u)$ and $E_2(u)$ are orthogonal. We shall make use of the natural partial order on the set of tripotents in a JB$^*$-triple. We write $u \leq e$ for tripotents $e$ and $u$ in a JB$^*$-triple $E$ if $e-u$ is a tripotent orthogonal to $u$, equivalently, $P_2(u)(e) = u$ (cf. \cite[Corollary 1.7]{FriedmanRusso1985}). \smallskip

We recall that a (Jordan) factor representation of a JB$^*$-algebra $\mathfrak{J}$  is a Jordan *-homomorphism $\pi$ from $\mathfrak{J}$ onto a weak$^*$-dense Jordan *-subalgebra of a JBW$^*$-algebra factor. Observe that every JB$^*$-algebra admits a faithful or separating family of Jordan factor representations (cf. \cite[Corollary 5.7]{AlfsenShultzStormer1978}). Since every Jordan $^*$-homomorphism is linear and preserves the Jordan product and the involution, it also preserves the associated Jordan triple product, i.e. is a \emph{Jordan Triple *-homomorphism} (see \cite{GarcesPeralta2013}). 

For further information on the classical theory of Jordan--Banach algebras and JB$^*$-triples the reader is referred to \cite{AlfsenShultz2003}, \cite{EscolanoPeraltaVillena2025},\cite{HOS},\cite{Kaup1983}. 

\subsection{Spin Factors and Type \texorpdfstring{$I_2$}{I2} JBW\texorpdfstring{$^*$}{*}-algebras} \label{subsec:spin-factors}\ \\

In this subsection, we review the structure theory of type $I_2$ JBW$^*$-algebras and spin factors, following \cite{HOS, Stacey1982}. We begin by recalling the general notion of type $I_n$ summands in the setting of JBW$^*$-algebras.\smallskip

A $JBW^*$-algebra $\mathfrak{J}$ is said to be of \emph{type $I_n$} (for a cardinal $n \ge 1$) if there exists a family $\{p_\alpha\}_{\alpha \in \Lambda}$ of orthogonal abelian projections in $\mathfrak{J}$ with cardinality $|\Lambda| = n$ such that each projection has central support $c(p_\alpha) = \unit$ and $\sum_{\alpha \in \Lambda} p_\alpha = \unit$ \cite[5.3.3]{HOS}. In particular, $\mathfrak{J}$ is of \emph{type $I_2$} if its identity element $\unit$ decomposes as the sum of two orthogonal abelian projections with central support $\unit$.\smallskip

In the associative setting, the classical example of a type $I_2$ factor is the von Neumann algebra $M_2(\mathbb{C})$, consisting of all $2 \times 2$ matrices with complex entries. When $M_2(\mathbb{C})$ is endowed with the symmetrized Jordan product, it becomes a JBW$^*$-factor of type $I_2$. Nevertheless, if we consider $S_2(\mathbb{C})$, the space of $2 \times 2$ complex symmetric matrices, it forms a type $I_2$ JBW$^*$-factor under the Jordan product, whereas it is not even closed under the usual associative product. This shows that the class of type $I_2$ factors in the Jordan setting is strictly wider than in the associative framework. Indeed, \cite[Theorem 6.1.8]{HOS} establishes that the class of type $I_2$ factors in $JBW^*$-algebras coincides precisely with the class of \emph{spin factors}.\smallskip

Recall that given a complex Hilbert space $H$, a \emph{spin factor} is a norm-closed subspace $V$ of $B(H)$ such that $\dim(V) > 2$, $a^* \in V$, and $a^2 \in \mathbb{C}\unit$, for all $a \in V$, where $\unit$ denotes the unit of $B(H)$.  Furthermore, it is known that there exists an inner product $\langle \cdot | \cdot \rangle $ on $V$ satisfying 
$ab^* + b^*a = 2\langle a | b \rangle, \hbox{for all } a,b \in V$. The norm $\|a\|_2^2:=\langle a | a \rangle$ given by the inner product is equivalent to the operator norm $\| \cdot \|$, and both are related by 
$$\|a\|^2 = \|a\|_2^2 + (\|a\|_2^4 - |\langle a | a \rangle|^2)^{\frac{1}{2}}.  $$
For each norm-one element $u \in V$ with $u^*= u$, the Banach space $V$ is a unital JBW$^*$-algebra with unit $u$, Jordan product $a\circ_ub = \frac{1}{2}(aub + bua)$ and involution $a^{*_u}:=ua^*u$ (c.f. \cite[page 358]{Harris1981}). It is worth to note that if $V$ is a spin factor, the self-adjoint part of the associated JBW$^*$-algebra $V_{sa}= \{a \in V \, : \, a^{*_u} = a $ is a real spin factor in the sense of \cite[\S 6]{HOS}. \smallskip

Spin factors are also called \emph{Type IV Cartan factors}, i.e. a type of JBW$^*$-triple. Indeed, given a spin factor $V$, under the triple product defined by 
$$ \{a,b,c\} = \langle a|b \rangle c + \langle c|b\rangle a - \langle a|c^*\rangle b^*$$
for any $a,b,c \in V$, it has the structure of  JBW$^*$-triple. Moreover, as any Cartan factor,  $V$ is an \emph{atomic} JBW$^*$ -triple. i.e. the linear span of all its minimal tripotents is weak*-dense in $V$.  It is known (see \cite[Table 1 in page 210]{Kaup1997} that every spin factor of dimension bigger than one has rank $2$. Therefore, given a non-zero minimal tripotent $u\in V$, we can choose a non-zero minimal tripotent $v\in V$, orthogonal to $u$, such that $e=u+v$ is a complete tripotent in $V$.\smallskip

To deal with general type $I_2$ JBW$^*$-algebras and manage type $I_2$ summands in the bidual, we rely on the structure theory developed by P.~J.~Stacey \cite{Stacey1982}. \\

\section{Quasi-linear functionals on JB\texorpdfstring{$^*$}{*}-algebras}

As it was said previously, \emph{positive quasi-linear functionals} were originally introduced by J. F. Aarnes to study the linearity of a physical state on an arbitrary C$^*$-algebra. Our starting point then is the definition of what a quasi-linear functional on the setting of JB$^*$-algebras is.  

\begin{definition}\label{def: quasi_lin}
    Let $\mathfrak{J}$ be a JB$^*$-algebra where $J$ stands for the self-adjoint part of the algebra. A \emph{ quasi-linear Jordan functional} on $J$ is a mapping $\mu : J \rightarrow \mathbb{R}$ such that, whenever $B$ is an associative Jordan subalgebra of $J$, the restriction of $\mu$ to $B$ is linear. Furthermore, $\mu$ is bounded on the closed unit ball of $J$. 
\end{definition}

Observe now that given any quasi-linear Jordan functional $\mu$ on $J$, we can define a functional $\overline{\mu}: \mathfrak{J} \rightarrow \mathbb{C}$ such that $\overline{\mu}(x +iy) = \mu(x) + i \mu(y)$ whenever $x, y \in J$. Then, $\overline{\mu}$ will be linear on each maximal associative Jordan *-subalgebra of $\mathfrak{J}$. Nevertheless, under the hypothesis of the Definition \ref{def: quasi_lin}, this extension is not unique, there could exist $\nu: \mathfrak{J} \rightarrow \mathbb{C}$ such that $\nu$ is linear on each maximal associative Jordan *-subalgebra of $\mathfrak{J}$, $\nu$ agrees with $\mu$ in $J$, but in general $\nu  \neq \overline{\mu}$ in $\mathfrak{J}$. We will see that, under the hypothesis of \Cref{Theo: 2.3}, the uniqueness of the extension it is recovered.

\begin{remark}

It is well known that $M_2(\mathbb{C})$, regarded as a unital $C^*$-algebra, exhibits pathological behaviour with respect to quasi-linear functionals: there exist quasi-linear functionals on $M_2(\mathbb{C})$ which fail to be linear \cite{BunceWright1996}.

Note that $M_2(\mathbb{C})$ itself is not a JB-algebra, but its self-adjoint part $M_2(\mathbb{C})_{sa}$, equipped with the \emph{natural Jordan product} $a \circ b = \frac{1}{2}(ab+ba)$, is a canonical example of a JB-algebra. In this setting, the same obstruction is already present. Indeed same  counterexamples work here. 

This behaviour is not restricted to finite-dimensional cases. In the broader framework of JBW$^*$-algebras, the same structural phenomenon persists: the existence of quasi-linear functionals that are not linear is precisely linked to the presence of direct summands of type $I_2$.
Conversely, as shown in \cite{EscolanoPeraltaVillena2025} in the context of the MGBW-problem, if a JBW$^*$-algebra $\mathfrak{J}$ contains no direct summand of type $I_2$, then every quasi-linear functional on $\mathfrak{J}$ is automatically linear.
\end{remark}

For the purpose of the paper, we now present a weaker definition of quasi-linearity. 

\begin{definition}
Under the same hypothesis in  Definition \ref{def: quasi_lin}, a \emph{local quasi-linear Jordan functional} is a function $\mu : J \rightarrow \mathbb{R}$ such that, for each $x \in J$, $\mu$ is linear on the smallest norm-closed Jordan subalgebra of $J$ containing $x$. Furthermore, $\mu$ is bounded on the closed unit ball of $J$.
\end{definition}

It is clear that each quasi-linear Jordan functional on $J$ is a local quasi-linear Jordan functional, although the converse is false, even in the associative case. The counterexample given by J. F. Aarnes in \cite{Aarnes1991} shows that there exist an abundance of non-linear local quasi-functional on $C(X)$, where $X$ is some compact Hausdorff space. The construction remains valid in the Jordan setting since $C(X)$, endowed with its natural Jordan product, is a JB-algebra where every associative Jordan subalgebra here coincides with the commutative subalgebras of $C(X)$.  

For this same reason, the result by C. A. Akemann and S. M. Newberger in \cite{AkemannNew1973} holds in the associative Jordan setting for the same class of spaces. Akemann and Newberger's result responds to a gap that they identified in the proof of \cite[Lemma 9]{Aarnes1969}; rather than repairing that proof, they gave an independent argument establishing the result whenever $X$ is a totally disconnected locally compact Hausdorff space and the algebra $C_0(X)$ of all real-valued continuous functions on $X$ which vanish at infinity is considered. \smallskip

We now present the technical core of the present paper. 

\begin{proposition}\label{Prop: mu_bar_extension}
   Let $\mathfrak{J}$ be a JB$^*$-algebra, $\mu: J \rightarrow \mathbb{R}$ a local quasi-linear Jordan functional, and $M = J^{**}$ the bidual of $\mathfrak{J}$. Assume additionally that $\mu$, when restricted to the closed unit ball of $J$, is uniformly continuous with respect to the weak topology $\sigma(J, J^*)$. Then the following holds:
\begin{enumerate}[$(a)$]
    \item there exists a unique extension of $\mu$, $\overline{\mu}: M \rightarrow \mathbb{R}$, which is homogeneous, bounded and uniformly weak$^*$-continuous when restricted to the closed unit ball of $M$;
    \item $\overline{\mu}$ is a local quasi-linear Jordan functional on $M$.
\end{enumerate}
\end{proposition}

\begin{proof}
    \noindent $(a)$: By Goldstine's theorem, the closed unit ball $J_1$ is weak-* dense in the closed unit ball $M_1$. Hence, for any $x \in M_1$, there exists a net $\{a_{\lambda}\} \subset J_1$ that converges to $x$ in the weak-* topology. Now observe that, by construction, the restriction of the weak-* topology of $M$ to the subspace $J$ coincides with the weak topology on $J$. Thus, $\{a_{\lambda}\}$ is a Cauchy net in the weak topology of $J$.\smallskip
    
    By hypothesis, the restriction of $\mu$ to $J_1$ is uniformly weakly continuous. Therefore, $\{\mu(a_{\lambda})\}$ is a Cauchy net in $\mathbb{R}$, and hence its limit exists and is well defined. Now define the map $\overline{\mu}_1: M_1 \rightarrow \mathbb{R}$ by $\overline{\mu}_1(x) = \lim_{\lambda} \mu(a_{\lambda})$. Since $\mu$ is uniformly weakly continuous, this mapping is well defined. Furthermore, by classical topology (see \cite{Kelley1953}), $\overline{\mu}_1$ is unique and uniformly weak-* continuous on $M_1$. Moreover, since $\mu$ is bounded on $J_1$, $\overline{\mu}_1$ is bounded on $M_1$.\smallskip 

To extend $\mu$ to $M$, we define 
$$\bar{\mu}(x) = \begin{cases}
       0, & \text{if } x = 0, \\ 
       \|x\|\bar{\mu}_1\left( \frac{x}{\|x\|} \right), & \text{if } x \neq 0.
       \end{cases}
$$
By hypothesis, $\mu$ is a local quasi-linear Jordan mapping on $J$; in particular, $\mu(ta) = t\mu(a)$ for every $a \in J$ and $t \in \mathbb{R}^*$. Since scalar multiplication is weak-* continuous,  the net $\{-a_{\lambda}\}$ converges in the weak-* topology to $-x$, and hence $$\overline{\mu_1}(-x) = \lim_{\lambda}\mu(-a_{\lambda}) =-\lim_{\lambda}\mu(a_{\lambda}) = -\overline{\mu_1}(x).$$ Therefore, it suffices to consider scalars $t>0$ to prove the homogeneity of $\overline{\mu}$.  Indeed, for any non-zero element $x \in M$ and $t \in \mathbb{R}^+$, we have 
\begin{equation*}
    \overline{\mu}(tx) = \| tx\| \overline{\mu}_1 \left( \frac{tx}{\|tx\|}\right) = t\|x\| \overline{\mu}_1 \left(  \frac{x}{\|x\|}\right) = t\overline{\mu}(x),
\end{equation*}
as desired.

\noindent $(b)$: by part $(a)$ we already have that $\overline{\mu}: M \rightarrow \mathbb{R}$ is bounded when restricted to $M_1$. We shall now prove that is linear on every norm closed Jordan-subalgebra generated by a single element in $M$, i.e. for every $x \in M$, $\overline{\mu}$ is linear when restricted to $J(x)$, the norm-closed Jordan real subalgebra generated by $x$. Since $\overline{\mu}$ is homogeneous, it is enough to prove it for every $x \in M_1$.\smallskip

First observe that by the Bipolar theorem, the weak-* closure of any convex subset $K$ of $M$ coincides with the strong closure of $K$ in $M$. Since $J_1$ is convex and weakly-* dense in $M_1$ by Goldstine, we get that $J_1$ is strongly dense in $M_1$.

Now, take any element $x \in M_1$. There exists a net $\{a_{\lambda}\} \subset J_1$ that strongly converges to $x$. By \cite[4.1.9]{HOS}, Jordan multiplication is jointly strongly continuous on bounded sets of a JBW-algebra, so that for any $k \in \mathbb{N}$ we get that $\{a_{\lambda}^k\}$ strongly converges to $x^k$. Now let $\phi_1$ and $\phi_2$ be polynomials with real coefficients and zero constant term. It is clear that the polynomial $\phi_1 + \phi_2$ is of the same class,  
and hence $\phi_1(a_\lambda) \to \phi_1(x)$, $\phi_2(a_\lambda) \to \phi_2(x)$, and $(\phi_1 + \phi_2)(a_\lambda) \to (\phi_1 + \phi_2)(x)$ strongly in $M_1$.\smallskip

By $(a)$, $\overline{\mu}$ when restricted to $M_1$ is weakly-* continuous. Then, since the strong topology is finer that the weak-* topology, $\overline{\mu}$ is strongly continuous in $M_1$ and hence $\lim_{\lambda}{\mu}(a_{\lambda}^k) = \bar{\mu}(x^k)$. In particular, since $\mu$ is a local quasi-linear Jordan functional,
$$\mu (\phi_1(a_{\lambda})) + \mu(\phi_2(a_{\lambda})) = \mu((\phi_1 + \phi_2)(a_{\lambda})),$$
so that
$$ \overline{\mu}(\phi_1(x)) + \overline{\mu}(\phi_2(x)) = \overline{\mu}(\phi_1(x) + \phi_2(x)).$$

Now consider $J(x)$, the Jordan real subalgebra generated by $x$ in $M$. By construction of $J(x)$, it is the norm-closure of the set of all elements of the form $\phi(x)$, where $\phi$ is a polynomial with real coefficients and zero constant term. Then, since each norm convergent sequence is bounded and strongly convergent, $\overline{\mu}$ is linear in $J(x)$ for every $x \in M_1$, as we wanted to prove. 
\end{proof}

Our purpose now shall be to apply the MGBW theorem in \cite{EscolanoPeraltaVillena2025} when possible on the second dual of the JB-algebra we are studying. To this end, we should be sure in which case the bidual of $J$ has no type $I_2$ direct summands.

\begin{remark} \label{rem:why-all-spins}
It is well known that $M_2(\mathbb{C})$ carries non-linear quasi-functionals \cite{Aarnes1970}. In the solution to the quasi-linearity problem for $C^*$-algebras \cite{BunceWright1996}, excluding quotients isomorphic to $M_2(\mathbb{C})$ is sufficient because all type $I_2$ direct summands in von Neumann algebras are of the form $C(X, M_2(\mathbb{C}))$. In the Jordan setting, however, \cite{Stacey1982} reveals that type $I_2$ summands in the JBW$^*$-algebra $\mathfrak{J}^{**}$ can involve higher- or infinite-dimensional spin factors. Therefore, one must exclude quotients isomorphic to any spin factor $V$ to guarantee that the bidual $\mathfrak{J}^{**}$ possesses no type $I_2$ direct summands.
\end{remark}

Next proposition solves this issue commented in the previous Remark.

\begin{proposition}\label{prop:no_I2}
Let $\mathfrak{J}$ be a JB*-algebra. If $\mathfrak{J}$ has no quotient isomorphic to any spin factor, then
the JBW*-algebra $\mathfrak{M} = \mathfrak{J}^{**}$ has no direct summand of type $I_2$.
\end{proposition}

\begin{proof}

Suppose to the contrary that $M$ has a direct summand of type $I_2$. Thus, by \cite[Corollary to Theorem 1, pp.~124--125]{Stacey1982} and \cite{Edwards1980}, there exists a factor representation $\pi: M \rightarrow V$, where $\pi(M)$ is a weak*-dense Jordan *-subalgebra in $V$. Since the canonical embedding $i_{\mathfrak{J}}: \mathfrak{J} \hookrightarrow M$ has weak*-dense image and $\pi$ is weak*-continuous, the  mapping $\Phi:= \pi\circ i_{\mathfrak{J}}: \mathfrak{J} \rightarrow V $ has weak*-dense image. Moreover, $\Phi$ is a Jordan *-homomorphism, and in particular, a triple homomorphism. \smallskip 

Since $\ker \Phi$ is a closed ideal of $\mathfrak{J}$, by \cite{Wright1977} the quotient $\mathfrak{J}/ \ker\Phi$ is again a JB$^*$-algebra, and hence, a JB$^*$-triple. Then, consider the induced quotient mapping $\hat{\Phi}: \mathfrak{J}/ \ker\Phi \rightarrow V$.  Since $\Phi(\mathfrak{J})$ is weak*-dense in $V$ and $\Phi(\mathfrak{J}) = \hat{\Phi}(\mathfrak{J}/ \ker\Phi)$, we conclude that $\hat{\Phi}(\mathfrak{J}/ \ker\Phi)$ is also weak*-dense in $V$. Therefore $\hat{\Phi}$ is a Jordan *-monomorphism whose image is  weak*-dense JB$^*$-subtriple of $V$ (cf. \cite{Kaup1983}). Thus, $\hat{\Phi}(\mathfrak{J}/ \ker\Phi)$ can be ragarded as a weak*-dense JB$^*$-subtriple of $V$.\smallskip

Fix a non-zero minimal tripotent $u \in V$. Since $V$ is a rank-two JBW$^*$-triple, there exists  a non-zero orthogonal minimal tripotents $v \in V$ such that $e:= u +v$ is a complete tripotent in $V$. By  \cite [Theorem 3.3]{BuncePoloPeralta2006} (or \cite[Theorem 3.8]{PoloPeralta2010}) applied to $\hat{\Phi}(\mathfrak{J}/ \ker\Phi)$ there exist norm-one pairwise orthogonal elements $b_1, b_2 \in \hat \Phi(\mathfrak{J}/ \ker \Phi)$ such that  $ \{u,u,b_1\} = u$ and $\{v,v,b_2\} = v$  in $V$. Therefore, by the Pierce decomposition relative to $u$ and $v$, this gives 
$$
\begin{aligned}
    b_1 = u + P_0(u)(b_1), \quad \text{and } \quad   b_2 = v + P_0(v)(b_2).
\end{aligned}
$$
The elements $b_1$ and $b_2$ have norm-one, therefore, let $r(b_1)$ and $r(b_2)$ denote the range tripotents of $b_1$ and $b_2$ respectively. Since $P_2(u)(b_1) = u$ and $P_2(v)(b_2) = v$, we have  $u  \leq r(b_1)$ and $v \leq r(b_2)$, i.e. $u \in V_2(r(b_1))$ and $v \in V_2(r(b_2))$ (cf. \cite[Proposition 2.4]{HamhalterKalendaPeralta2020}. Moreover, $b_1 \perp b_2$ implies $r(b_1) \perp r(b_2)$. Since $b_1 \in V_2(r(b_1))$ and $b_2 \in V_2(r(b_2))$, it follows that
$$
\begin{aligned}
    b_1 \in E_2(r(b_1))\subset E_0(r(b_2)), \quad \text{and } \quad
    b_2 \in E_2(r(b_2)) \subset E_0(r(b_1))
\end{aligned}
$$
Moreover, by \cite[Lemma 1.14(1)]{Horn1987}
$$
\begin{aligned}
    E_0(r(b_1)) \subset E_0(u) \quad \text{and } \quad  E_0(r(b_2)) \subset E_0(v),
\end{aligned}
$$
respectively, which implies that $b_1 \in E_0(v)$ and $b_2 \in E_0(u)$. As $u\perp v$, we have $u \in V_0(v)$ and $v \in V_0(u)$. Consequently,
$$
\begin{aligned}
    P_0(u)(b_1) = b_1 - u \in E_0(v) \quad \text{and } \quad   P_0(v)(b_2) = b_2 - v \in E_0(u).
\end{aligned}
$$
As a consequence of the previous computations, we conclude that the elements $P_0(u)(b_1)$, $ P_0(v)(b_2) \in E_0(u)\cap E_0(v)$.  Since $e=u+v$ is a complete tripotent and $u \perp v$,  this shows that $0 = E_0(u+v) = E_0(u)\cap E_0(v)$, and thus, that $P_0(u)(b_1) = 0 =P_0(v)(b_2)$. Therefore, $b_1 = u$ and $b_2 = v$. In particular, $u$ belongs to $\hat{\Phi}(\mathfrak{J}/ \ker\Phi)$. Since $u$ was an arbitrary minimal tripotent of the spin factor $V$,  all minimal tripotents of $V$ belong to $\hat{\Phi}(\mathfrak{J}/ \ker\Phi)$. Since $V$ is atomic, we conclude that 
$ \hat{\Phi}(\mathfrak{J}/ \ker\Phi)= V$. This contradicts the initial hypothesis that $\mathfrak{J}$ has no quotient isomorphic to a spin factor. 
\end{proof}

With Proposition~\ref{prop:no_I2} established, our next goal is to construct a bounded finitely additive measure on the lattice of projections of the bidual JBW-algebra $M = J^{**}$. This will allow us to invoke the general Mackey--Gleason--Bunce--Wright theorem for JBW$^*$-algebras in \cite{EscolanoPeraltaVillena2025} to obtain a bounded linear extension on $M$. For the sake of completeness, we recall the standard definition of a finitely additive measure on the lattice of projections of a JBW$^*$-algebra. 

\begin{definition}
Let $\mathfrak{J}$ be a JBW$^*$-algebra and let $\mathcal{P}(\mathfrak{J})$ denote its orthomodular lattice of projections. A mapping $\psi : \mathcal{P}(\mathfrak{J}) \to \mathbb{C}$ is called a \emph{bounded finitely additive measure} if it satisfies the following two conditions: 
\begin{enumerate}[$(a)$]
\item $\psi(p + q) = \psi(p) + \psi(q)$ whenever $p, q \in \mathcal{P}(\mathfrak{J})$ are orthogonal projections; 
\item $\sup \{ |\psi(p)| : p \in \mathcal{P}(\mathfrak{J}) \} < \infty$.
\end{enumerate}
\end{definition}

Finally, we are in a position to present the main result of this paper.

\begin{theorem}\label{Theo: 2.3}
    Let $\mathfrak{J}$ be a unital JB$^*$-algebra with no quotient isomorphic to any spin factor. Let $\mu$ be a locally quasi-linear Jordan functional on $J$. Then $\mu$ is a bounded linear functional if and only if the restriction of $\mu$ to the closed unit ball of $J$ is uniformly weakly continuous.
\end{theorem}

\begin{proof}
    By \cite[Lemma 1.1]{BunceWright1996}, each bounded linear functional on $J$ is uniformly weakly continuous. We now assume that $\mu$ is uniformly weakly continuous in $J_1$.\smallskip

    By Proposition \ref{Prop: mu_bar_extension}, there exists a unique extension of $\mu$, $\overline{\mu}: M \rightarrow \mathbb{R}$, where $M = J^{**}$, such that $\overline{\mu}$ is a homogeneous, bounded local quasi-linear Jordan functional on $M$. 
    Now let us take $p_1, p_2, \dots p_n$ orthogonal projections in the JBW-algebra $M$. We define 
    $$x = p_1 + \frac{1}{2}p_2 + \dots + \frac{1}{2^{n-1}}p_n + \frac{1}{2^n}(\unit-p_1-p_2-\dots -p_n).$$
    Then, $\{x^k\}_{k \in \mathbb{N}}$, converges in norm to $p_1$. So $p_1 \in J(x)$. Then, $\{(2x-2p_1)^{k}\}_{k \in \mathbb{N}}$ converges in norm to $p_2$. Similarly, $p_3, p_4, \dots, p_n$ and $\unit - p_1 -p_2 - \dots -p_n$ are all in $J(x)$. Hence, $\overline{\mu}$ is finitely additive. 

    Let us define the mapping $\psi : \mathcal{P}(M) \rightarrow \mathbb{R}$ by 
    $\psi(p) = \overline{\mu}(p)$. Then, by construction, $\psi$ is a bounded finitely additive measure on the projections of  the JBW-algebra $M$. By hypothesis $\mathfrak{J}$ has no quotients isomorphic to any spin factor. Thus, by Proposition \ref{prop:no_I2}, $\mathfrak{M}=\mathfrak{J}^{**}$ has no direct summands of type $I_2$, so that $M$ has no direct summands of type $I_2$ (c.f. \cite{Edwards1980}). Hence, by the Mackey-Gleason-Bunce-Wright theorem in \cite{EscolanoPeraltaVillena2025}, $\psi$ extends to a bounded linear functional on $\mathfrak{M}$ (we will denote this extension again by $\psi$), such that coincides with $\overline{\mu}$ in  $\mathcal{P}(M)$.
    
     By \cite[Proposition 4.2.3]{HOS}, given  an element $x \in M$, for each positive $\varepsilon$ we can find mutually orthogonal projections $p_1, p_2 \dots , p_m \in W(x)$, the JBW-algebra generated by $x$ and $\unit$, and real numbers $\alpha_1, \dots \alpha_m$ such that 
    $$ \left\| x - \sum_{j=1}^m \alpha_j p_j \right\|< \frac{\varepsilon}{2 \hbox{ max}\{\|\bar{\mu}\|, \|\psi\|\}}. 
    $$
    From the previous paragraph, $\overline{\mu}$ and the extension $\psi$ coincide on finite real linear combinations of orthogonal projections, which implies that 
    $$ \overline{\mu}\left( \sum_{j=1}^m \alpha_jp_j\right) = \sum_{j=1}^{m}\alpha_j\overline{\mu}(p_j) = \sum_{j=1}^{m}\alpha_j\psi(p_j) = \psi\left( \sum_{j=1}^m \alpha_jp_j\right).$$
    Following the same computations in \cite[Theorem 3.4]{EscolanoHamhalterPeraltaVillena2026} we get that  $\overline{\mu}(x) = \psi(x)$ for all $x \in M$. Thus, $\mu$ is linear and bounded as we wanted to prove. 
\end{proof}

As a consequence, the extension $\overline{\mu}: \mathfrak{J} \rightarrow \mathbb{C}$ given by $\overline{\mu}(x +iy) = \mu(x) + i \mu(y)$ is unique when we impose that the restriction of $\mu$ to the closed unit ball of $J$ is uniformly weakly continuous.

\medskip\medskip

\textbf{Acknowledgements}\medskip \ \\
 Supported by grant FPU21/00617 at University of Granada founded by Ministerio de Universidades (Spain), 
and by the IMAG--Mar{\'i}a de Maeztu grant CEX2020-001105-M/AEI/10.13039/ 501100 011033.\smallskip

The author would also like to thank Professor Antonio M. Peralta for his valuable recommendations and helpful comments, which greatly improved the quality of this manuscript.

\smallskip\smallskip

\noindent\textbf{Data Availability} Statement Data sharing is not applicable to this article as no datasets were generated or analyzed during the preparation of the paper.\smallskip\smallskip

\noindent\textbf{Declarations}
\smallskip\smallskip

\noindent\textbf{Conflict of interest} The author declares that he has no conflict of interest.

\end{document}